%% file: main.tex
\documentclass{amsart}
\usepackage{graphicx} %
\usepackage{url}
\usepackage{amsthm}
\usepackage{amsmath}
\usepackage{amsfonts}
\usepackage{amssymb}
\usepackage{MnSymbol}
\usepackage{enumerate}
\usepackage{extarrows}
\usepackage{comment}
\usepackage{booktabs}
\usepackage{graphicx}

\usepackage{color}
\definecolor{darkgreen}{rgb}{0,0.5,0}
\usepackage[colorlinks, citecolor=darkgreen, backref]{hyperref}

\definecolor{darkred}{rgb}{0.5,0,0}

\theoremstyle{plain}
\newtheorem{theorem}{Theorem}
\numberwithin{theorem}{section}
\newtheorem{lemma}[theorem]{Lemma}

\newtheorem{cor}[theorem]{Corollary}
\newtheorem{prop}[theorem]{Proposition}

\newtheorem*{lemma*}{Lemma}
\theoremstyle{definition}

\newtheorem{example}[theorem]{Example}

\newcommand{\N}{\mathbb{N}}
\newcommand{\Z}{\mathbb{Z}}
\newcommand{\Q}{\mathbb{Q}}

\newcommand{\co}{\mathcal{O}_K}
\newcommand{\eps}{\varepsilon}
\newcommand{\disc}{\mathrm{disc}}
\newcommand{\Gal}{\mathrm{Gal}}

\newcommand{\uu}{\eps}
\newcommand{\vv}{\delta}
\newcommand{\numcyclicsporadic}{17}
\newcommand{\numcomplexsporadic}{14}
\newcommand{\NK}{\mathcal{N}_K}
\newcommand{\NKbound}{d \cdot 24^{324(2d+2)} + 2}

\begin{document}

\title[]{Sums of two units in number fields}

\begin{abstract}
Let $K$ be a number field with ring of integers $\co$.  Let $\NK$ be the set of positive integers $n$ such that there exist units $\uu, \vv \in \co^\times$ satisfying $\uu + \vv = n$.  We show that $\NK$ is a finite set if $K$ does not contain any real quadratic subfield.  In the case where $K$ is a cubic field,  we also explicitly classify all solutions to the unit equation $\uu + \vv = n$ when $K$ is either cyclic or has negative discriminant.
\end{abstract}

\author{Magdal\'{e}na Tinkov\'{a}}
\address{Faculty of Information Technology \\ Czech Technical University in Prague \\ Th\'{a}kurova 9 \\ 160 00 Praha 8 \\ Czech Republic}
\email{tinkova.magdalena@gmail.com}

\author{Robin Visser}
\address{Charles University \\ Faculty of Mathematics and Physics \\ Department of Algebra \\ Sokolovsk\'{a} 83 \\ 186 75 Praha 8 \\ Czech Republic}
\email{robin.visser@matfyz.cuni.cz}

\author{Pavlo Yatsyna}
\address{Charles University \\ Faculty of Mathematics and Physics \\ Department of Algebra \\ Sokolovsk\'{a} 83 \\ 186 75 Praha 8 \\ Czech Republic}
\email{p.yatsyna@matfyz.cuni.cz}

\date{\today}
\thanks{M. T. was supported by Czech Science Foundation GA\v{C}R, grant 22-11563O. R. V. and P. Y. were both supported by {Charles University} programme 
PRIMUS/24/SCI/010 and %
{Charles University} programme UNCE/24/SCI/022}

\maketitle

\section{Introduction}
Let $K$ be a number field with ring of integers $\co$.  The motivation for the present paper is to provide a classification of the integers $n\in \Z$ and units $\uu,\vv \in \co^\times$ satisfying
\begin{equation} \label{eq:1}
     \uu + \vv = n.
\end{equation}
For any number field $K$, one always has the trivial solutions $1+1 = 2$,  $-1-1=-2$, and $u - u = 0$ for all $u \in \co^\times$. 
A well-known theorem of Siegel \cite{Siegel} asserts that, for any given number field $K$ and nonzero integer $n \in \Z$, there are only finitely many units $\uu, \vv \in \co^\times$ such that $\uu + \vv = n$, with Baker's \cite{Baker} results on linear forms in logarithms providing effective algorithms to compute all such units $(\uu,\vv)$ over any number field $K$.  
In the case where $n = 1$, effectively determining all unit solutions to $\uu + \vv = 1$ (and more generally all $S$-unit solutions) has numerous well-known Diophantine applications, e.g. computing all elliptic curves or hyperelliptic curves over $K$ with good reduction outside a given finite set of primes $S$ \cite{Koutsianas, vonKanel}, computing all solutions to Thue and Thue--Mahler equations \cite{TzanakisDeWeger1, TzanakisDeWeger2}, and computing polynomials with given discriminant \cite{EvertseGyory_disc, Smart}.  See also \cite{AKMRVW, EvertseGyory_unit, EGST, FreitasSiksek} for many other applications to solving unit equations.

For a fixed number field $K$, let $\NK$ denote the set of positive integers $n \in \Z$ such that there exist units $\uu, \vv \in \co^\times$ satisfying $\uu + \vv = n$.  By applying a theorem of Szemer\'edi \cite{Szemeredi}, it is known that $\NK$ is a density zero subset of the positive integers $\N$ (e.g. see \cite[Corollary~6]{JardenNarkiewicz}), with an asymptotic polylogarithmic upper bound shown by Fuchs--Tichy--Ziegler \cite{FuchsTichyZiegler}.  
Newman \cite[p.~89]{Newman} posed the question of determining $\NK$ for prime cyclotomic fields $K = \Q(\zeta_p)$, observing that $\{1, 2, 3\} \subseteq \mathcal{N}_{\Q(\zeta_p)}$ for all primes $p > 3$.  Kostra \cite{Kostra} and Newman \cite{Newman93} independently showed that $kp \nin \mathcal{N}_{\Q(\zeta_p)}$ for all primes $p \geq 3$ and all integers $k \in \Z$.  Jarden--Narkiewicz \cite[p.~331, Problem~C]{JardenNarkiewicz} also posed the general problem of determining an asymptotic formula for $N_2(x) := \# \{n \in \NK \;|\; n \leq x \}$ (see also \cite[p.~532, Problem~43]{Narkiewicz}).  It is clear that $\NK$ is finite for $K = \Q$ and imaginary quadratic fields $K$, however it can also be seen that $\NK$ is infinite for any real quadratic field $K$, e.g. if $\eps_K$ denotes the fundamental unit of $K$, then for all integers $i \geq 1$, we have $\mathrm{Tr}_{K/\Q}(\uu_K^i) \in \NK$, due to the unit solution $\uu_K^i + \overline{\uu}_K^i = \mathrm{Tr}_{K/\Q}(\uu_K^i)$.  

This naturally poses the more general question of classifying all the number fields $K$ for which $\NK$ is finite.  Similarly, we can also ask about classifying the number fields $K$ for which $n \in \NK$ for some fixed positive integer $n$.

\subsection{Case $n=1$} Classifying solutions to the unit equation $\uu + \vv = 1$ (and more generally classifying number fields $K$ such that $1 \in \NK$) is a well-studied problem.  Following the terminology introduced by Nagell \cite{Nagell}, we say a unit $u \in \co^\times$ is \emph{exceptional} if $1-u$ is also a unit in $\co$. In a series of several papers, Nagell \cite{Nagell28, Nagell59, Nagell60, Nagell64, Nagell68, Nagell69} studied and classified the number fields of unit rank zero or one containing such exceptional units.  Ennola \cite{Ennola} identified two infinite families of cubic number fields containing exceptional units. 
Lenstra \cite{Lenstra} constructed new examples of norm-Euclidean fields by proving that if a number field $K$ contains a sufficient number of exceptional units $\eps_1, \dots, \eps_m$ such that all pairwise differences $\eps_i - \eps_j$ are units, then $K$ is norm-Euclidean. Using Lenstra's criterion, Mestre \cite{Mestre}, Leutbecher--Martinet \cite{LeutbecherMartinet}, Leutbecher \cite{Leutbecher}, Leutbecher--Niklasch \cite{LeutbeckerNiklasch}, and Houriet \cite{Houriet} constructed many new examples of norm-Euclidean fields.

Numerous other properties and applications of exceptional units have also been extensively studied by Freitas--Kraus--Siksek \cite{FreitasKrausSiksek}, Louboutin \cite{Louboutin1, Louboutin2}, Niklasch--Smart \cite{NiklaschSmart}, Silverman \cite{Silverman95, Silverman96}, Stewart \cite{Stewart12, Stewart13} and Triantafillou \cite{Triantafillou}.  Further background on exceptional units can be found in \cite[Section~5.5]{EvertseGyory_unit} and the references therein.

\subsection{Case $n=2$} Since $1 + 1 = 2$, we trivially have that $2 \in \NK$ for all number fields $K$.  However, one can instead ask for which number fields $K$ it is possible to express 2 as the sum of two \emph{distinct} units.  A related problem is classifying the number fields $K$ for which every element in $\co$ can be expressed as a finite sum of \emph{distinct} units in $\co^\times$; these are the so-called Distinct Unit Generated (DUG) fields. We remark this is one of the open problems listed in Narkiewicz \cite[p.~530, Problem~18]{Narkiewicz}. Jacobson  \cite{Jacobson} was one of the first to classify such fields, proving that $\Q(\sqrt{2})$ and $\Q(\sqrt{5})$ are DUG fields. Belcher \cite{Belcher} proved that if a number field $K$ has a non-trivial unit solution to $\uu + \vv = 2$ and has the property that all elements in $\co$ can be expressed as finite sums of (not necessarily distinct) units in $\co^\times$, then $K$ is a DUG field.  This criterion has been successfully applied to compute many examples of DUG fields, e.g. see Dombek--Mas\'akov\'a--Ziegler \cite{DombekMasakovaZiegler},  Filipin--Tichy--Ziegler \cite{FilipinTichyZiegler}, Hajdu--Ziegler \cite{HajduZiegler}, and Ziegler \cite{Ziegler}.

For further applications and surveys on additive unit representations in number fields, see also Barroero--Frei--Tichy \cite{BFT}, Filipin--Tichy--Ziegler \cite{FilipinTichyZiegler2}, Frei \cite{Frei}, Jarden--Narkiewicz \cite{JardenNarkiewicz}, and Tichy--Ziegler \cite{TichyZiegler}.  We also refer the reader to \cite[Section~10.2]{EvertseGyory_unit} and the references therein.

\bigskip

The first aim of the present paper is to show that $\NK$ is finite for all number fields $K$ not containing a real quadratic subfield.

\begin{theorem} \label{thm:NKfinite}
    Let $K$ be a degree $d$ number field not containing any real quadratic subfield.  Then $\NK$ is a finite set and satisfies the bound $|\NK| \leq \NKbound$.
\end{theorem}

Our proof of this theorem makes essential use of the finiteness of solutions to a unit equation in several unknowns, and thus these methods cannot effectively determine the set $\NK$.  Nevertheless, in the particular case where $K$ is a cubic number field, we can explicitly classify all unit solutions to equation (\ref{eq:1}) over all cyclic cubic fields $K$ and all complex cubic fields $K$, for all integers $n$.  We recall a cubic field $K$ is \emph{cyclic} if $\mathrm{Gal}(K/\Q) \cong C_3$ (equivalently, if $K/\Q$ is abelian), and \emph{complex} if $K$ has signature $(1,1)$ (equivalently, if the discriminant of $K$ is negative).

\subsection{Cyclic cubic fields} A particular well-known family of cyclic cubic number fields are the simplest cubic fields $K_a$, famously studied by Shanks \cite{S}. For each integer $a \in \Z$, we define $K_a = \Q(\rho_a)$ where $\rho_a$ is a root of the cubic polynomial $f_a(x) := x^3 - ax^2 - (a+3)x - 1$.  It can be shown that both $\rho_a$ and $\rho_a + 1$ are units in $\co^\times$, and thus the equation (\ref{eq:1}) always has the trivial solution $-\rho_a + (\rho_a + 1) = 1$ for any integer $a \in \Z$.  We note that $K_a = K_{-a-3}$, thus we can without loss of generality restrict to the case $a \geq -1$.
We also remark that all pairs of integers $(a, b)$ such that $K_a \cong K_b$ have been classified by Hoshi \cite[p.~2137]{Hoshi}; in particular we have the isomorphisms $K_{-1} = K_5 = K_{12} = K_{1259}$, $K_0 = K_3 = K_{54}$, $K_1 = K_{66}$, and $K_{2} = K_{2389}$.

Recently, Vukusic--Ziegler \cite{VZ} showed that in the family of Shanks' simplest cubics, a solution to equation (\ref{eq:1}) for $\uu,\vv \in \Z[\rho_a]^\times$ is either one of the above-mentioned  trivial solutions or one of finitely many sporadic solutions, assuming $n \leq \max(|a|^{1/3}, 1)$.  They conjectured that this bound on $n$ is not a necessary assumption. We extend their results by proving the following general classification of solutions to the unit equation $\uu + \vv = n$ over all cyclic cubic number fields $K$.

\begin{theorem} \label{thm:maincyclic}
Let $K$ be a cyclic cubic number field, $n \in \mathbb{Z}$, and $\uu, \vv \in \co^\times$ such that $\uu + \vv = n$. Then exactly one of the following three cases holds:
\begin{enumerate}
    \item Either $n = 0$, or $(\uu,\vv) = (1,1)$ or $(\uu,\vv) = (-1,-1)$.

    \item $K = K_a$ for some integer $a \geq -1$, $n  = \pm 1$, and either $(\uu, \vv) = (-n \sigma(\rho_a), n \sigma(\rho_a + 1))$ or $(\uu, \vv) = (-n \sigma(\rho_a + 1), n \sigma(\rho_a))$ for some $\sigma \in \mathrm{Gal}(K/\Q)$.

    \item $K = K_a$ for some $a \in \{-1, 0, +1\}$ and $(\uu, \vv)$ is equivalent to one of the $\numcyclicsporadic$ sporadic solutions listed in Table \ref{tab:cyclicsolutions}.
\end{enumerate}
In particular, $\NK$ is finite for all cyclic cubic fields $K$, and is given by
\begin{equation*}
    \NK = \begin{cases}
        \{ 1, 2, 3, 4, 5, 19, 22 \} & \text{if } K = K_{-1}, \\
        \{1, 2, 3 \}  & \text{if } K = K_{0}, \\
        \{1, 2, 5, 7 \}  & \text{if } K = K_{1} , \\
        \{1, 2\} & \text{if } K \text{simplest cubic}, K \nin \{K_{-1}, K_0, K_1\}, \\
        \{2\}   & \text{otherwise}. 
    \end{cases}
\end{equation*}

\end{theorem}

Here, we say that a pair of solutions $(\uu_1, \vv_1)$ and $(\uu_2, \vv_2)$ are \emph{equivalent} if either $(\uu_1, \vv_1) = (\pm \sigma(\uu_2), \pm \sigma(\vv_2))$ or $(\uu_1, \vv_1) = (\pm \sigma(\vv_2), \pm \sigma(\uu_2))$ for some $\sigma \in \mathrm{Gal}(K/\Q)$.

We also mention some recent work of Komatsu \cite{Komatsu} who have also classified all solutions to the unit equation $\uu + \vv = n$ over all cyclic cubic number fields $K$ using similar methods. The sets $\NK$ for each cyclic cubic field $K$ are essentially given in \cite[Theorem~5.1]{Komatsu}.

In particular, our main theorem implies that no solution to $\uu+\vv = n$ exists for any integer $n$ such that $|n| > 22$, over any cubic cyclic field $K$, and we also note that our theorem resolves both Conjecture 1 and Conjecture 2 proposed by Vukusic--Ziegler \cite[p.~717]{VZ}.

Our strategy is to compute the minimal polynomials of all the pairs $(\uu,\vv)$ of solutions to the equation (\ref{eq:1}). By studying the discriminant of these polynomials, we are able to find conditions as to when the corresponding solutions are in cyclic cubic fields. This is transformed into a question about integer solutions to a family of binary cubic equations (and in turn to Thue equations), for which we can apply some results of Hoshi \cite{Hoshi}.

\subsection{Complex cubic fields} We also prove an analogous theorem for complex cubic fields.  For each nonzero integer $a \geq -1$, we define $L_a := \Q(\omega_a)$ where $\omega_a$ is a root of the cubic polynomial $g_a(x) := x^3 - ax^2 - 1$.  This time, we have that both $\omega_a$ and $\omega_a - a$ are units in $\co^\times$, and thus equation (\ref{eq:1}) always has the trivial solution $\omega_a + (-\omega_a + a) = a$ for any nonzero $a \geq -1$.  Analogously to the cyclic cubic fields, we also show that these are the only infinite family of solutions to equation (\ref{eq:1}) amongst the complex cubic fields.

\begin{theorem} \label{thm:maincomplex}
    Let $K$ be a complex cubic number field, $n \in \Z$, and $\uu, \vv \in \co^\times$ such that $\uu + \vv = n$. Then exactly one of the following three cases holds:
    \begin{enumerate}
        \item Either $n = 0$, or $(\uu,\vv) = (1,1)$ or $(\uu,\vv) = (-1,-1)$.

        \item $K = L_a$ for some nonzero integer $a \geq -1$, $n = \pm a$, and either $(\uu, \vv) = (\pm \sigma(\omega_a), \pm \sigma(-\omega_a+a))$ or $(\uu, \vv) = ( \pm \sigma(-\omega_a+a), \pm \sigma(\omega_a))$ for some $\sigma \in \mathrm{Gal}(K/\Q)$.

        \item $K$ is either $L_{-1}$, $L_1$, or $\Q[x]/(x^3 - x^2 - x-1)$, and $(\uu, \vv)$ is equivalent to one of the $\numcomplexsporadic$ sporadic solutions listed in Table \ref{tab:complexsolutions}.
    \end{enumerate}
\end{theorem}

In particular, an explicit computation of the set $\NK$ for some complex cubic fields $K$ can be given by
\begin{equation} \label{eq:NKcomplex}
    \NK = 
    \begin{cases}
           \{1, 2, 3, 4\} & \text{if } K = L_{-1}, \\
           \{1, 2, 3, 67 \}   & \text{if } K = L_1, \\
           \{2, a\}  & \text{if } K = L_a \text{ for some positive } a \leq 1000 \text{ with } a \nin \{1, 67\}, \\
           \{2\} & \text{if } K \neq  L_a \text{ for all } a \geq -1 .
    \end{cases}
\end{equation}

The paper is structured as follows.   In Section~\ref{sec:nkbounds}, we give a proof of Theorem~\ref{thm:NKfinite} using known bounds on the number of solutions to unit equations in three unknowns. 
In Section~\ref{sec:minpoly}, we derive general expressions for the possible minimal polynomials of units $\uu, \vv \in \co^\times$ which satisfy $\uu + \vv = n$ over some cubic number field $K$.  In Section~\ref{sec:disc}, we recall some basic properties about discriminants of algebraic integers and explicitly compute the discriminants of the minimal polynomials derived in the previous section.  In Section~\ref{sec:cyclicubic}, we give a proof of Theorem~\ref{thm:maincyclic}, applying a deep result of Hoshi \cite{Hoshi}.  In Section~\ref{sec:complexcubic}, we give a proof of Theorem~\ref{thm:maincomplex}. Finally in Section~\ref{sec:openproblems}, we pose some open problems about the sets $\NK$.

\bigskip
We are very grateful to Volker Ziegler for valuable discussions and for communicating to us a proof of the finiteness of $\NK$. We thank Ingrid Vukusic for helpful comments on an earlier draft of this paper. We also thank the referees for a careful reading of the paper and for many useful comments, including a suggestion which significantly improved the upper bound in Theorem~\ref{thm:NKfinite} compared to an earlier draft of this paper.

\section{Finiteness of $\NK$} \label{sec:nkbounds}

In this section, we shall prove that $\NK$ is finite for all number fields $K$ not containing any real quadratic subfield.   For completeness, we first consider the trivial case where the unit group of $K$ has rank zero.

\begin{lemma} \label{lem:easyNK}
    If $K = \Q$, then $\NK = \{2\}$, and if $K$ is an imaginary quadratic field, then
    \begin{equation*}
        \NK = \begin{cases}
            \{1, 2\} &\text{if } K = \Q(\sqrt{-3}), \\
            \{2\} &\text{otherwise. }\\
        \end{cases}
    \end{equation*}
\end{lemma}

\begin{proof}
    Clearly if $K = \Q$, then $\co^\times = \{ \pm 1\}$ and so $\NK = \{2\}$.  If $K = \Q(\sqrt{-3})$, then $\co^\times = \{\pm \omega, \pm \omega^2, \pm 1\}$ where $\omega = \frac{1 + \sqrt{-3}}{2}$.  A standard check gives the only solutions to (\ref{eq:1}) are equivalent to $\omega - \omega^2 = 1$ or the trivial solution $1 + 1 = 2$, thus $\NK = \{1, 2\}$.  If $K = \Q(\sqrt{-1})$, then $\co^\times = \{\pm \sqrt{-1}, \pm 1\}$, where we easily observe that only the trivial solution $1 + 1 = 2$ exists.  Otherwise, if $K$ is any other imaginary field, then $\co^\times = \{ \pm 1\}$ and so again only the trivial solution $1+1 = 2$ exists.
\end{proof}

Whilst $\NK$ is infinite for real quadratic fields $K$, we can still give a complete classification of the set $\NK$ in terms of the solutions to certain Pell equations.

\begin{prop}
    Let $K = \Q(\sqrt{D})$ be a real quadratic field, with $D \geq 2$ a positive squarefree integer. Then
    \begin{equation} \label{eq:NKrealquad}
        \NK = \{ x \in \Z_{\geq 1} \;|\; x^2 - Dy^2 = \pm 4 \text{ for some } y \in \Z \} .
    \end{equation}
    In particular, for all $n \neq 2$, $n \in \NK$ if and only if $K = \Q(\sqrt{n^2 \pm 4})$.
\end{prop}

\begin{proof}
    First, let $n \in \NK$. %
    Then there exist $\uu, \vv \in \co^\times$ such that $\uu + \vv = n$.  Let $f_{\uu}(x) = x^2 - ax \pm 1$ be the minimal polynomial of $\eps$, where $a$ is the trace of $\eps$ over $K$. Let $\overline{\eps}$ denote the conjugate of $\eps$ in $K$. As $f_\eps(x) = (x - \eps) ( x- \overline{\eps})$, we note that we have
    \begin{equation*}
        n^2 - an \pm 1 = f_\eps(n) = (n - \eps) (n - \overline{\eps}) = \delta \overline{\delta} = \pm 1 .
    \end{equation*}
    Thus $n^2 - an$ is either $0$ or $ \pm 2$. 
    
    We first consider the case where $n^2 - an = \pm 2$. Therefore $n$ divides $2$, and so either $n = 1$ or $n = 2$.   We observe that $2$ is always an element of the right-hand side of (\ref{eq:NKrealquad}), as $(x,y) = (2, 0)$ is always a solution to $x^2 - Dy^2 = 4$.  If $n = 1$, then $1 - a = \pm 2$, which implies $a = -1$ or $a = 3$.  We thus consider the two cases:
    \begin{itemize}
        \item \textbf{Case $a = -1$}. We have $f_\eps(x) = x^2 + x \pm 1$.  As $K$ is real quadratic, this implies $\eps = (-1 \pm \sqrt{5})/2$, and thus $D = 5$.  As $(x,y) = (1,1)$ is a solution to $x^2 - Dy^2 = -4$, this proves $n = 1$ is an element of the right-hand side of (\ref{eq:NKrealquad}).

        \item \textbf{Case $a = 3$}. We have $f_\eps(x) = x^2 - 3x \pm 1$, which implies either $\eps = (3 \pm \sqrt{5})/2$ or $\eps = (3 \pm \sqrt{13})/2$.  In the latter case, an easy check confirms that $1 - \eps$ is not a unit in $\Q(\sqrt{13})$, thus we must have $\eps = (3 \pm \sqrt{5})/2$.  Therefore by the same argument as the case for $a = -1$, we have that $n = 1$ is an element of the right-hand side of (\ref{eq:NKrealquad}).
    \end{itemize}

    We now consider the case where $n^2 - an = 0$. As $n \neq 0$, this implies $n = a = \text{Tr}_{K/\Q}(\eps)$.  Now if $\eps = (x + y \sqrt{D})/2$ for some integers $x, y$, then since $\eps$ is a unit, this implies $x^2 - D y^2 = \pm 4$, which again proves that $n$ is an element of the right-hand side of (\ref{eq:NKrealquad}).

    Conversely, let $n$ be a positive integer such that $n^2 - Dy^2 = \pm 4$ for some $y \in \Z$.  Define $\eps := (n + y \sqrt{D})/2$.  As $\eps$ has integral trace $n$ and norm $\pm 1$, this implies $\eps \in \co^\times$.   Therefore, as $n = \eps + \overline{\eps}$, this implies $n \in \NK$, which proves the theorem.
\end{proof}

Now in order to bound $|\NK|$ for arbitrary number fields $K$ which do not contain real quadratic subfields, we shall make essential use of the following 
finiteness result for unit equations in three unknowns, using an explicit upper bound on the number of such solutions shown by Amoroso--Viada \cite{AmorosoViada}.

\begin{theorem}[{\cite[p.~412]{AmorosoViada}}]
\label{thm:generalunit}
    Let $K$ be a number field, and let $\Gamma$ be a subgroup of $(K^\times)^3$ of finite rank $r$.  Then the number of solutions $(x_1, x_2, x_3) \in \Gamma$ such that $x_1, x_2, x_3 \neq 1$ and
    \begin{equation*}
        x_1 + x_2 + x_3 = 1 
    \end{equation*}
    is finite and can be bounded above by 
    $24^{324(r+4)}$.
\end{theorem}

\subsection*{Proof of Theorem~\ref{thm:NKfinite}}
    Let $K$ be a degree $d$ number field with a set of fundamental units $\eps_1, \dots, \eps_r$.  We assume $K$ contains no real quadratic subfield. Let $L$ be the Galois closure of $K/\Q$ and let $G = \Gal(L/\Q)$.  Let $\uu, \vv \in \co^\times$ and $n$ be a positive integer such that $\uu + \vv = n$.  
    
    We aim to use Theorem \ref{thm:generalunit} to show that there are only finitely many possibilities for $n$.  For any $\sigma \in G$, we have that $\uu + \vv = n = \sigma(n) = \sigma(\uu) + \sigma(\vv)$.  By dividing through by $\sigma(\vv)$, we therefore obtain the following three term unit equation,
    \begin{equation} \label{eq:uniteq}
        \frac{\uu}{\sigma(\vv)} +  \frac{\vv}{\sigma(\vv)} - \frac{\sigma(\uu)}{\sigma(\vv)} = 1 .
    \end{equation}

    For each $\sigma \in G$, we let $\Gamma_\sigma$ denote the subgroup of $(L^\times)^3$ of finite rank $2r$ generated by the following $2r+1$ elements of $(L^\times)^3$,

    \begin{equation} \label{eq:Gamma_gens}
        \Big\{  \big(\eps_i, \, 1, \, \sigma(\eps_i) \big) \Big\}_{i = 1, \dots, r} \cup \, \Big\{  \big( \tfrac{1}{\sigma(\eps_i)}, \, \tfrac{\eps_i}{\sigma(\eps_i)}, \, \tfrac{1}{\sigma(\eps_i)} \big) \Big\}_{i = 1, \dots, r} \cup \, \big\{ (1, \, 1, \, -1) \big\} .
    \end{equation}
    
    Recall that by Theorem \ref{thm:generalunit}, there exists a finite set $\mathcal{C}_\sigma \subseteq (L^\times)^3$ with the bound $|\mathcal{C}_\sigma| \leq 24^{324(2r+4)}$  such that if $(\alpha, \beta, \gamma) \in \Gamma_\sigma$ satisfies $\alpha,\beta,\gamma\neq 1$ and $\alpha + \beta + \gamma = 1$, then $(\alpha, \beta, \gamma) \in \mathcal{C}_\sigma$.
    We note that the set $\mathcal{C}_{\sigma}$ depends only on the subgroup $\Gamma_{\sigma}$ and in particular on the $2r+1$ generators given in (\ref{eq:Gamma_gens}).  This implies $\mathcal{C}_{\sigma}$ depends only on the action of $\sigma$ on $K$, i.e. the restricted automorphism $\sigma|_K$.  As there are only $d$ distinct such maps $\sigma|_K$, this implies there are at most $d$ distinct sets $\mathcal{C}_\sigma$ amongst all of $\{ \mathcal{C}_\sigma \}_{\sigma \in G}$. 

    We now consider the following two cases:

    \begin{itemize}
        \item \textbf{Case 1.} There exists some $\sigma \in G$ such that none of the three terms on the left-hand side of (\ref{eq:uniteq})  equals 1.  Therefore, applying Theorem \ref{thm:generalunit}, for such a $\sigma$, we have that there exists $(\alpha, \beta, \gamma) \in \mathcal{C}_\sigma$  such that 
        $$\frac{\uu}{\sigma(\vv)} = \alpha, \quad \frac{\vv}{\sigma(\vv)} = \beta, \;\text{ and }\;  -\frac{\sigma(\uu)}{\sigma(\vv)} = \gamma.$$
        In particular, this implies that $n = \uu + \vv = \sigma(\vv) (\alpha + \beta)$, and so $n^{\deg(L/\Q)} = \mathrm{N}_{L/\Q}(n) = \mathrm{N}_{L/\Q}(\alpha + \beta)$. Therefore $n$ is uniquely determined by the triple $(\alpha, \beta, \gamma)$.  As $|\mathcal{C}_\sigma| \leq 24^{324(2r+4)}$ and $r \leq d - 1$, this implies there are at most $ | \bigcup_{\sigma \in G} \mathcal{C}_\sigma | \leq d \cdot 24^{324(2d+2)}$ such possibilities for $n$. 

        \item \textbf{Case 2.}  For each $\sigma \in G$, at least one of the three terms $\uu/\sigma(\vv)$, $\vv/\sigma(\vv)$, or $- \sigma(\uu)/\sigma(\vv)$ must equal 1. If $- \sigma(\uu)/\sigma(\vv) = 1$ then $\uu + \vv = 0$, which contradicts $n$ being a positive integer. Therefore, this means that $\sigma(\vv) \in \{\uu, \vv\}$ for all $\sigma \in G$.  Thus, the minimal polynomial of $\vv$ must have degree at most 2, and so $\Q(\vv)$ is either $\Q$ or an imaginary quadratic field.  In particular, $n$ must therefore be a sum of two units in $\Q$ or an imaginary quadratic field, and thus by Lemma~\ref{lem:easyNK}, we have $n \in \{1, 2\}$.
    \end{itemize}

    Combining both cases, this implies there are at most $\NKbound$ positive integers $n$ such that there exist units $\uu, \vv \in \co^\times$ satisfying $\uu + \vv = n$. This proves the theorem.
\qed %

\bigskip
\noindent \textbf{Remarks.}
\begin{itemize}
    \item We note that Theorem \ref{thm:NKfinite} gives an explicit bound on the size of $\NK$, however the methods used in the proof crucially rely on the ineffective finiteness results for the number of solutions of unit equations in three unknowns.  We therefore pose the open question of whether an effective algorithm to determine $\NK$ exists.

    \item After publishing an initial preprint of this theorem, we received communication from Volker Ziegler giving an alternate method to proving Theorem~\ref{thm:NKfinite}.  In particular, a special case of a theorem of Levesque and Waldschmidt \cite[Theorem~3.3]{LevesqueWaldschmidt} implies that, for any field $K$ of degree $d \geq 3$, the twisted Thue equations of the form $|\mathrm{Nm}_{K/\Q}(X - \eps Y)| = 1$ have only finitely many solutions over all $(X, Y, \eps) \in \Z^2 \times \co^\times$ such that $XY \neq 0$ and $K = \Q(\eps)$. 
    Note that, if $n \in \NK$, then there exists $\eps \in \co^\times$ such that $|\mathrm{Nm}_{K/\Q}(n - \eps)| = 1$. Thus by applying \cite[Theorem~3.3]{LevesqueWaldschmidt} in the special case where $X = n$ and $Y = 1$, this implies $\NK$ is finite.  As their methods also rely heavily on Schmidt's subspace theorem, this similarly cannot yield an effective bound on the height of $n$.

    \item It seems reasonable to conjecture that our computed upper bound of $\NKbound$ for $|\NK|$ is still far from the true maximal value of $|\NK|$ over all degree $d$ number fields.  If we let $n_d$ denote the maximum value of $|\NK|$ over all degree $d$ number fields not containing a real quadratic subfield, we can therefore ask about the true asymptotic behaviour of $n_d$ as $d \to \infty$.
    
    We can give a quadratic lower bound for $n_d$ %
    via the following construction.   For each $d \geq 2$, we construct the number field $M_d := \Q(\alpha_d)$, where $\alpha_d$ is a root of the defining polynomial
    \begin{equation*}
        (x - 2)(x-4) \cdots (x - 2^d) - 1.
    \end{equation*}
    Thus, since $\prod_{i=1}^d (\alpha_d - 2^i) = 1$, this implies that $\alpha_d - 2^i$ are units in $\mathcal{O}_{M_d}^\times$ for all $i =1, \dots, d$.  Therefore, we have the unit solutions $(\alpha_d - 2^j) - (\alpha_d  - 2^i) = 2^i - 2^j$, and so $2^i(1 - 2^{j-i}) \in \mathcal{N}_{M_d}$ for all pairs $1 \leq i < j \leq d$.  This gives the lower bound $n_d \geq |\mathcal{N}_{M_d}| \geq \binom{d}{2} = \frac{d(d-1)}{2}$.

\end{itemize}

\section{The parameterized solutions over cubic fields} \label{sec:minpoly}

For the remainder of this paper, we shall restrict to the case where $K$ is a cubic number field.
For a cubic number field $K$, let $\sigma_i:K \hookrightarrow \mathbb{C}$ be the three embeddings of $K$, for $i = 1, 2, 3$. Let $\mathrm{N}_{K/\Q}(\alpha)=\prod_{i=1}^3\sigma_i(\alpha)$ be the absolute norm of $\alpha \in K$. 

If $n = 0$, then clearly all solutions are given by $(\uu,\vv) = (u, -u)$ for any $u \in \co^\times$.  If $\uu,\vv \in \Q$, then the only possible solutions are $(\uu,\vv) = (1,1)$, $(1,-1)$, $(-1,1)$ or $(-1,-1)$. 
Also, if $n < 0$, then a solution $(u_1, u_2)$ to $\uu + \vv = n$ gives rise to a solution $(-u_1, -u_2)$ to $\uu + \vv = -n$.  Therefore, we may assume for the rest of the paper that $n > 0$ and that $\uu,\vv \nin \Q$. 

We can rewrite equation (\ref{eq:1}) as $\uu=n-\vv$. 
 By taking the absolute norms on both sides we find that 
\[\pm 1= \mathrm{N}_{K/\Q}(\uu) = \mathrm{N}_{K/\Q}(n-\vv)=\prod^3_{i=1}(n-\sigma_i(\vv))=f_\vv(n),\] 
where $f_\vv(x)\in\Z[x]$ is the minimal polynomial of $\vv$. Furthermore, we have that 
\begin{equation}\label{eq:2}
    (-1)^3f_\vv(n-x)=\prod^3_{i=1}(x-(n-\sigma_i(\vv)))=\prod^3_{i=1}(x-\sigma_i(\uu))=f_\uu(x).
\end{equation}
where $f_\uu(x)\in\Z[x]$ is the minimal polynomial of $\uu$.
Let \[f_\vv(x)=x^3+ax^2+bx\pm 1,\] where $a,b\in\Z.$ From the above, it follows that 
\begin{equation} \label{eq:3}
    n^3+an^2+bn = -f_{\uu}(0) - f_{\vv}(0) = 0,\pm 2.
\end{equation} 

We now consider the two different cases $n^3+an^2+bn \neq 0$ and $n^3+an^2+bn = 0$, where the nature of our solutions for the unit equation $\uu + \vv = n$ depend crucially on which of these two cases we are in.

If $n^3+an^2+bn$ is nonzero, then $n$ divides $\pm2$, and thus we can consider the following two cases for $n$: 

\begin{itemize}
    \item \textbf{Case $n=1$}: Equation (\ref{eq:3}) implies that $1 + a + b = \pm 2$, and therefore either $a + b = 1$ or $a + b = -3$, depending on whether $f_\vv(0)=- 1$ or $+1$ respectively.

    \item \textbf{Case $n=2$}: Similarly, equation (\ref{eq:3}) implies that  $4+2a+b=\pm 1$, and thus $b = -2a -3$ or $b = -2a -5$, depending on whether $f_\vv(0)=- 1$ or $+1$ respectively.
\end{itemize}

Let us apply equation (\ref{eq:2}) and summarize the above:

\begin{lemma} \label{lemma:n12}
If $n^3+an^2+bn \neq 0$, then $n = 1$ or $n = 2$ and the minimal polynomials $f_\uu(x)$ and $f_\vv(x)$ take the following form for some integer $a \in \Z$:
\begin{enumerate}
    \item If $n=1$ we have 
    \begin{alignat*}{5}
        f_\uu(x) &= x^3-(a+3)x^2+(a+4)x-1 &\text{ and } f_\vv(x) &= x^3+ax^2-(a-1)x-1 , \\
        \text{or } f_\uu(x)&=x^3-(a+3)x^2+ax + 1 &\text{ and } f_\vv(x)&=x^3+ax^2-(a+3)x+1.
    \end{alignat*}
    
    \item If $n=2$ we have 
    \begin{alignat*}{5}
    f_\uu(x) &= x^3-(a+6)x^2+(2a+9)x-1& \text{ and } f_{\vv}(x) &= x^3+ax^2-(2a+3)x-1 , \\ 
    \text{ or } f_\uu(x) &= x^3-(a+6)x^2+(2a+7)x + 1 &\text{ and } f_{\vv}(x) &= x^3+ax^2-(2a+5)x+1.
    \end{alignat*}
    
\end{enumerate} 
\end{lemma}

Let us assume now that  $n^3+an^2+bn = 0$ and so $f_\uu(0)=-f_\vv(0)$. Thus, $n$ is a root of $x^2+ax+b$, and in particular, $a=-n-n_\vv$ and $b=nn_\vv,$ for some $n_\vv \in \Z.$ Without loss of generality, assume that $f_\vv(0)=1$, then $f_\uu(0)=-1$. 
Given that $f_\vv(x)=x^3-(n+n_\vv)x^2+nn_\vv x+1$ we see that 
\begin{align}
    f_\uu(x) = -f_{\vv}(n-x)
    &= - (n-x)^3 + (n + n_\vv) (n-x)^2 - n n_\vv (n-x) - 1 \nonumber  \\
    &= x^3-(2n-n_\vv)x^2+(n^2-nn_\vv)x-1 \label{eq:pre4} \\
    &= x^3-(n+n_\uu)x^2+nn_\uu x-1, \label{eq:4}
\end{align}
where $n_\uu := n - n_\vv$.  This result, which is essentially a reformulation of a proposition given originally by Minemura \cite[Proposition~1(a)]{Minemura}, %
can be summarized in the following lemma:

\begin{lemma}\label{lemma:n>}
If $n^3+an^2+bn = 0$ and $\uu, \vv \in \co^\times$ is a solution to $\uu+\vv=n$, then $\uu$ and $\vv$ have minimal polynomials
\begin{equation*}
    f_\uu(x)=x^3-(n+n_\uu)x^2+nn_\uu x-1 \text{ and } f_\vv(x)=x^3-(n+n_\vv)x^2+nn_\vv x+1,
\end{equation*}
respectively, where $n=n_\uu+n_\vv$ for some integers $n_\uu,n_\vv\in \Z.$
\end{lemma}

In order to classify all solutions $(\uu, \vv)$ to the unit equation $\uu + \vv = n$, it thus suffices to classify the possible integers $a \in \Z$ in Lemma~\ref{lemma:n12} and the possible integers $n_\uu, n_\vv\in \Z$ in Lemma~\ref{lemma:n>}.

\section{Discriminants} \label{sec:disc}

Let us recall that the \emph{discriminant} of a degree $d$ monic polynomial $f(x) \in \Z[x]$ is given by $\prod_{1 \leq i < j \leq d} (\alpha_i - \alpha_j)^2$ where $\alpha_1, \dots, \alpha_d$ are the $d$ complex roots of $f$. We denote this as $\disc(f)$. We also define the discriminant of an algebraic integer $\alpha \in \co$ as the discriminant of its minimal polynomial over $K$.

We begin by stating the following well known results (e.g. see Shanks \cite{S}):

\begin{lemma}
Let $K$ be a cyclic cubic number field, then the discriminant of $K$ is a square in $\Z.$ 
\end{lemma}

\begin{lemma}
    Let $K$ be a complex cubic number field, then the discriminant of $K$ is a negative integer.
\end{lemma}

\begin{lemma}
Let $\alpha \in \co$. Then the discriminant of $\alpha$ is a square multiple of the discriminant of $K.$
\end{lemma}

Putting these three lemmas together gives the following:
\begin{cor}\label{cor:mem}
If $K$ is a cyclic cubic number field, then the discriminant of any element $\alpha \in \co$ is an integer square.  If $K$ is a complex cubic number field, then the discriminant of any element $\alpha \in \co$ is negative.
\end{cor}

Now, let us compute the discriminants of the minimal polynomials in the previous section. We begin by examining polynomials from Lemma \ref{lemma:n12}. We have:

\begin{lemma} \label{lemma:n12_disc}
For each pair of polynomials $f_\uu(x)$ and $f_\vv(x)$ in Lemma \ref{lemma:n12}, we have the following discriminants:
\begin{enumerate}
    \item If $n=1$, then 
    \begin{align*}
        \disc(f_{\uu}) &= \disc(f_{\vv}) = a^4 + 6a^3 + 7a^2 - 6a - 31 \\
        \text{or } \disc(f_{\uu}) &= \disc(f_{\vv}) = (a^2+3a+9)^2.
    \end{align*}
    \item If $n=2$, then 
        \begin{align*}
             \disc(f_{\uu}) &= \disc(f_{\vv}) =(4a^2 + 24a + 9)(a+3)^2 \\
             \text{or } \disc(f_{\uu}) &= \disc(f_{\vv}) = 4a^4 + 48a^3 + 229a^2 + 510a + 473.
        \end{align*}
\end{enumerate}

\end{lemma}

By similarly examining the polynomials from Lemma \ref{lemma:n>}, we have:

\begin{lemma} \label{lemma:n>_disc}
    If $n^3 + an^2 + bn = 0$, then $n = U$ and
    $$ \disc(f_\uu) = (UV)^2(U+V)^2 - 2(U-V)^3 + 6(U^3-V^3) - 27 $$ for some integers $U, V \in \Z$.
\end{lemma}

\begin{proof}
    Using equation (\ref{eq:pre4}), we substitute $U$ and $V$ for $n$ and $-n_\vv$ respectively to get $f_\uu(x) = x^3-(2U+V)x^2+(U+V)Ux-1$.  Taking the discriminant of this polynomial gives $(UV)^2(U+V)^2 - 2(U-V)^3 + 6(U^3-V^3) - 27$ as claimed.
\end{proof}

\section{The cyclic cubic fields} \label{sec:cyclicubic}

In this section, we now restrict to the case where $K$ is a cyclic cubic field, and hence has square discriminant.  We can use this condition with Corollary~\ref{cor:mem} to explicitly obtain the possible minimal polynomials $f_\uu(x)$ and $f_\vv(x)$, using both Lemma \ref{lemma:n12_disc} and Lemma \ref{lemma:n>_disc} to handle the two cases $n^3 + an^2 + bn \neq 0$ and $n^3 + an^2 + bn = 0$ respectively.

\begin{lemma} \label{lemma:n12_cyclic}
    Let $K$ be a cyclic cubic number field and $\uu, \vv \in \co^\times$ such that $\uu + \vv = n$ for some integer $n$. Then if $n^3 + an^2 + bn \neq 0$, then either $K = K_a$ and $(\uu,\vv)$ is equivalent to $(-\rho_a, \rho_a+1)$, or $a \in \{-10, -6, -5, -4, -2, 0, 2, 4\}$.
\end{lemma}

\begin{proof}
The conclusion will follow from Corollary~\ref{cor:mem}, that we need the discriminants given in Lemma~\ref{lemma:n12_disc} to equal a square. 

For $n = 1$, the discriminant of the first pair of minimal polynomials $f_\uu(x)$ and $f_\vv(x)$ is an irreducible polynomial in $a$ of degree greater than two, and thus by Siegel's Theorem, can only represent finitely many squares.  In particular, if $\disc(f_{\uu}) = k^2$ for some $k \in \Z$, then  by completing the square, we have that $(a^2 + 3a - 1)^2 - 32 = k^2$. Thus by a difference of squares, we obtain that $a^2 + 3a - 1 = \frac{d}{2} + \frac{16}{d}$ for some integer divisor $d$ of 32. A simple check yields the only possible integer solutions for $a$ are $a = -5$ or $2$.
For the second pair of minimal polynomials, these define Shanks' simplest cubic fields \cite{S}, where in particular we note that $f_\uu(x) = -x^3 f_a(-1/x)$ and $f_\vv(x) = - f_a(-x)$.  This therefore implies that $(\uu, \vv)$ is equivalent to $(-\rho_a, \rho_a+1)$.

For $n=2$, we proceed by a similar argument to obtain a finite list of possible values for $a$.  In the first case, we must have that $(4a^2 + 24a + 9)$ is a square. By again completing a square we have that $(2a+6)^2-27=k^2$, for some $k\in \Z$. In particular, the only possible solutions are $a=-10,-6,0,4$.  In the second case, we again complete the square to obtain $(2a^2 + 12a + \frac{85}{4})^2 + \frac{343}{16} = k^2$ for some $k \in \Z$, and by another similar argument, we obtain that $8a^2 + 48a + 85 = \frac{d}{2} - \frac{343}{2d}$ for some integer divisor $d$ of 343. A simple check yields the only possible solutions for $a$ are $a = -4$ or $a = -2$.
\end{proof}

For the polynomial in Lemma~\ref{lemma:n>} we shall prefer the form given in equation~(\ref{eq:pre4}), substituting $U$ and $V$ for $n$ and $-n_{\vv}$ respectively:

\begin{lemma} \label{lemma:UVW}
Let $K$ be a cyclic cubic number field and $\uu, \vv \in \co^\times$ such that $\uu + \vv = n$ for some integer $n$. Assume that $n^3 + an^2 + bn = 0$. Then there exists an integer $W \in \Z$ such that 
\begin{equation} \label{eq:UVW}
    U^3 - WU^2V - (W+3)UV^2 - V^3 = W^2+3W+9
\end{equation}
\end{lemma}

\begin{proof}
Recall from Lemma~\ref{lemma:n>_disc}, that $\disc(f_\uu) = (UV)^2(U+V)^2 - 2(U-V)^3 + 6(U^3-V^3) - 27$. Thus, as $\disc(f_\uu)$ is a square, there exists some $W' \in \Z$ such that
 \[
 (UV)^2(U+V)^2 - 2(U-V)^3 + 6(U^3-V^3) - 27 = (UV(U+V)+W' )^2.
 \]
After some rearrangement and cancellations, this gives us 
\[
4U^3 + (6-2W')U^2V - (6 + 2W')UV^2 - 4V^3=W'^2+27.
\]
By a standard congruence check, we see that $W'$ is odd, thus we can substitute $W' := 2W+3$ for some $W \in \Z$. By furthermore dividing through by $4$, we get the desired equation:
    $U^3 - WU^2V - (W+3)UV^2 - V^3 = W^2+3W+9$.
\end{proof}

\begin{theorem}[{\cite[Corollary~1.6]{Hoshi}}] \label{thm:hoshi}
    There are only finitely many integer triples $(U,V,W)$ such that $U^3 - WU^2V - (W+3)UV^2 - V^3 = W^2 + 3W + 9$. In particular, the only integer solutions $(U,V)$ with $U > 0$ (for some $W \in \Z)$ are 
    \begin{gather}
        (1,-3), (1,4), (2,1), (2,5), (3,-22), (3,-1), (3,0), \label{eq:UVsols} \\
        (4,-5), (5,-7), (5,-4), (7, -5), (19,3), \text{ and } (22, -3). \nonumber
    \end{gather}
    
\end{theorem}

\begin{proof}
    Note that if we have a solution $(U,V,W)$, then we also have the solution $(-V, -U, -W-3)$.  Thus we may assume $W \geq -1$. The claim then follows immediately from Corollary 1.6 of Hoshi \cite{Hoshi}. 
\end{proof}

\noindent \textbf{Remarks}.

\begin{itemize}
    \item It is worth observing that equation (\ref{eq:UVW}) can be rewritten in the form $$\mathrm{N}_{K_W/\Q}(U - V \rho_W) = \sqrt{\disc(K_W)},$$
    where $K_W$ is the simplest cubic field with defining polynomial $x^3 - Wx^2 - (W+3)x - 1$.   
    As all the primes dividing $\disc(K_W)$ are ramified over $K_W$, this implies that there exists a  unique ideal $I = (\rho_W^2 + \rho_W + 1) \vartriangleleft \co$ such that $\mathrm{N}_{K_W/\Q}(I) = \sqrt{\disc(K_W)}$. Therefore, the equation~(\ref{eq:UVW}) is thus equivalent to finding integers $U, V \in \Z$ such that  $U - V \rho_W = u (\rho_W^2 + \rho_W + 1)$ for some unit $u \in \co^\times$. In the case where $\co^\times = \Z[\rho_W]^\times$, as considered in \cite{VZ}, it is known that $\{\rho_W, \rho_W+1 \}$ form a system of fundamental units for $\co^\times$ \cite{Thomas79}. Thus one can rewrite equation~(\ref{eq:UVW}) as $U - V \rho_W = \pm \rho_W^i (\rho_W+1)^j (\rho_W^2 + \rho_W + 1)$ for some integers $i, j \in \Z$, therefore reducing the problem to determining for which pairs of integers $(i,j)$ the $\rho_W^2$-coefficient of the expression  $\rho_W^i (\rho_W+1)^j (\rho_W^2 + \rho_W + 1)$ is zero.    \\

    \item We should also remark that the more general Thue equation  $\mathrm{N}_{K_W/\Q}(U - V \rho_W) = \lambda$ has been well-studied for many small values of $\lambda$.  Thomas \cite{Thomas90} and Mignotte \cite{Mignotte} classified all solutions to $\mathrm{N}_{K_W/\Q}(U - V \rho_W) = \pm 1$.  Mignotte--Peth\"{o}--Lemmermeyer \cite{MPL} and Lemmermeyer--Peth\"{o} \cite{LemmermeyerPetho} more generally classified all solutions to the Thue inequality $| \mathrm{N}_{K_W/\Q}(U - V \rho_W) | \leq  2W+3$.  Similar Thue inequalities were also studied by Lettl--Peth\"{o}--Voutier \cite{LettlPethoVoutier} and  Xia--Chen--Zhang \cite{XCZ}.
    
\end{itemize}

We can now prove our main Theorem~\ref{thm:maincyclic} classifying unit equations over cyclic cubic fields:

\subsection*{Proof of Theorem~\ref{thm:maincyclic}}
    Let $K$ be a cubic cyclic field, $n \in \Z$ and unit $\uu, \vv \in \co^\times$ such that $\uu + \vv = n$.  As before, if $n = 0$ or $\uu, \vv \in \Q$, then $(\uu, \vv)$ falls inside case (1) in the theorem statement.  

    Now assume $n \neq 0$ and $\uu, \vv \nin \Q$.  If $n^2 + an + b \neq 0$, then $f_\uu(x)$ and $f_{\vv}(x)$ are given in the form as shown in Lemma~\ref{lemma:n12}, with the possible values for $a$ given by Lemma~\ref{lemma:n12_disc}. A summary of the possible sporadic solutions in this case (excluding the trivial family of solutions equivalent to $(-\rho_a, \rho_a+1)$ over the simplest cubic fields $K_a$) is given in Table~\ref{tab:casen12}.

\begin{table}[h]
\centering
\caption{In the case $n^3 + an^2 + bn \neq 0$, we tabulate the possibilities for the  integer $n$, the integer $a$, and the corresponding minimal polynomial  $f_{\uu}(x)$  and number field $K$, given the possible pairs $(n,a)$ from Lemma~\ref{lemma:n12_cyclic}.  This table also includes the minimal polynomials for $\delta = n - \eps$; i.e. if $f(x)$ is listed with the integer $n$, then $-f(n-x)$ is also listed with the integer $n$.}
\label{tab:casen12}
\begin{tabular}{@{}cccc@{}}
\toprule
$n$ & $a$ & $f_{\uu}(x)$ & $K$ \\ \midrule

$1$ & $-5$ & $ x^3 + 2x^2 - x - 1$ &  $K_{-1}$  \\
$1$  & $2$  & $x^3 - 5x^2 + 6x - 1$  & $K_{-1}$ \\
$2$  &  $-10$   &  $x^3 + 4x^2 - 11x - 1$  &  $K_{-1}$  \\

$2$  &  $-6$  &   $x^3 - 3x - 1$   & $K_{0}$  \\

$2$  &  $-4$   &  $ x^3 - 2x^2 - x + 1$  &  $K_{-1}$ \\

$2$  &  $-2$   & $  x^3 - 4x^2 + 3x + 1$  &  $K_{-1}$  \\

$2$  &  $0$   &  $x^3 - 6x^2 + 9x - 1$  &  $K_0$  \\

$2$  &  $4$  &  $ x^3 - 10x^2 + 17x - 1$  &  $K_{-1}$   \\

\bottomrule
\end{tabular}%
\end{table}

    If $n^3 + an^2 + bn = 0$, then Lemma~\ref{lemma:n>} implies that $f_\vv(x) = x^3 - (n + n_\vv)x^2 + n n_\vv + 1$ for some $n, n_\vv \in \Z$.  Moreover, a combination of Lemma~\ref{lemma:UVW} and Theorem~\ref{thm:hoshi} implies that there are only finitely many possibilities for $(n, n_\vv)$ given the possibilities for $(U,V)$ listed in (\ref{eq:UVsols}).  We summarise these solutions in Table \ref{tab:UVs}.

\begin{table}[h]
\centering
\caption{In the case $n^3 + an^2 + bn = 0$, we tabulate the possibilities for the minimal polynomial $f_{\uu}(x)$, the integer $n$, and number field $K$, given the possible pairs $(U,V)$ from (\ref{eq:UVsols}).}
\label{tab:UVs}
\begin{tabular}{@{}cccc@{}}
\toprule
$(U, V)$ & $f_{\uu}(x)$ & $n$ & $K$ \\ \midrule

$(1, -3)$  &  $x^3 + x^2 - 2x - 1$  & $1$ & $K_{-1}$ \\
$(1, 4)$ &  $x^3 - 6x^2 + 5x - 1$   & $1$ & $K_{-1}$ \\

$(2, 1)$ &  $x^3 - 5x^2 + 6x - 1$  & $2$ &  $K_{-1}$\\
$(2, 5)$ &  $x^3 - 9x^2 + 14x - 1$  &  $2$  & $K_1$ \\
$(3, -22)$ &  $x^3 + 16x^2 - 57x - 1$   & $3$ & $K_{-1}$\\ 
$(3, -1)$ &  $x^3 - 5x^2 + 6x - 1$   & $3$ & $K_{-1}$ \\
$(3, 0)$ &  $x^3 - 6x^2 + 9x - 1$  & $3$ & $K_0$ \\
$(4, -5)$ &   $x^3 - 3x^2 - 4x - 1$ & $4$ & $K_{-1}$\\
$(5, -7)$ &  $x^3 - 3x^2 - 10x - 1$  & $5$ & $K_1$ \\ 
$(5,-4)$  &  $x^3 - 6x^2 + 5x - 1$  & $5$ & $K_{-1}$ \\
$(7, -5)$ &  $x^3 - 9x^2 + 14x - 1$  & $7$ &  $K_1$ \\ 
$(19, 3)$ &  $x^3 - 41x^2 + 418x - 1$  & $19$ & $K_{-1}$\\
$(22, -3)$ &  $x^3 - 41x^2 + 418x - 1$  & $22$ & $K_{-1}$ \\

\bottomrule
\end{tabular}%
\end{table}

Combining all solutions from both cases gives us the $\numcyclicsporadic$ equivalence classes of solutions shown in Table~\ref{tab:cyclicsolutions}, thus proving that $(\uu, \vv)$ falls inside either case (2) or (3) of the theorem statement.
\qed

\bigskip
\noindent \textbf{Remark}.
\begin{itemize}
    \item We note that all sporadic solutions lie in the three simplest cubic fields $K_{-1}$, $K_0$, and $K_1$. As all three fields are monogenic with unit group $\co^\times = \Z[\rho_a]^\times$, all such solutions were thus already found by Vukusic--Ziegler \cite{VZ}.  We also remark that equivalent sporadic solutions can be found in the simplest cubic fields $K_a$ for $a \in \{3, 5, 12, 54, 66, 1259\} $ as we have the isomorphic fields $K_{-1} = K_5 = K_{12} = K_{1259}$, $K_0 = K_3 = K_{54}$, and $K_1 = K_{66}$ (e.g. see \cite[p.~2137]{Hoshi}).

\end{itemize}

\subsection{Table of sporadic solutions}

Here, we give a full table listing all the non-trivial solutions to $\uu + \vv = n$ over all cyclic cubic fields $K$, shown in Table~\ref{tab:cyclicsolutions}.  For brevity, we only list solutions up to equivalence (as defined in \cite[p.~706]{VZ}).  We recall that a solution $(\uu, \vv)$ to the unit equation $\uu + \vv = n$ over a cubic cyclic field $K$ is \emph{non-trivial} if $n \neq 0$, $\uu, \vv \nin \Q$, and $(\uu, \vv)$ is not equivalent to $(-\rho_a, \rho_a + 1)$ for any $a \geq -1$.

\input{cyclic_solutions_table}

\section{The complex cubic fields} \label{sec:complexcubic}

In this section, we now restrict to the case where $K$ is a complex cubic field, and thus has negative discriminant. Using a strategy similar to the cyclic cubic case, we use Corollary~\ref{cor:mem} to explicitly obtain the possible minimal polynomials $f_\uu(x)$ and $f_\vv(x)$ whose discriminants are negative, again using both Lemma~\ref{lemma:n12_disc} and Lemma~\ref{lemma:n>_disc} to handle the two cases $n^3 + an^2 + bn \neq 0$ and $n^3 + an^2 + bn = 0$ respectively.

\begin{lemma} \label{lemma:n12_complex}
    Let $K$ be a complex cubic number field and assume $n^3 + an^2 + bn \neq 0$. Then either $n = 1$ and $a \in \{-4, -3, -2, -1, 0, 1\}$, or $n = 2$ and $a \in \{-5, -4, -2, -1\}$.
\end{lemma}

\begin{proof}
    For $n = 1$, we note that the polynomial $a^4 + 6a^3 + 7a^2 - 6a - 31$ is negative only if $a \in \{-4, -3, -2, -1, 0, 1\}$ and the polynomial $(a^2 + 3a + 9)^2$ is clearly never negative.
    
    For $n = 2$, we note that the polynomial $(4a^2 + 24a + 9)(a+3)^2$ is negative only if $a \in \{-5,-4, -2, -1\}$.  We can also check that the polynomial $4a^4 + 48a^3 + 229a^2 + 510a + 473$ has no real roots and is thus always positive.
\end{proof}

\begin{table}[h]
\centering
\caption{In the case $K$ is complex and  $n^3 + an^2 + bn \neq 0$, we tabulate the possibilities for the  integer $n$, the integer $a$, and the corresponding minimal polynomial  $f_{\uu}(x)$  and number field discriminant $\Delta_K$, given the possible pairs $(n,a)$ from Lemma~\ref{lemma:n12_complex}.  This table also includes the minimal polynomials for $\delta = n - \eps$; i.e. if $f(x)$ is listed with the integer $n$, then $-f(n-x)$ is also listed with the integer $n$.}
\label{tab:casen12complex}
\begin{tabular}{@{}cccc@{}}
\toprule
$n$ & $a$ & $f_{\uu}(x)$ & $\Delta_K$ \\ \midrule

$1$ & $-4$ & $x^3 + x^2 - 1$  & $-23$   \\
$1$ & $-3$ &  $x^3 + x - 1 $ & $-31$   \\
$1$ & $-2$ &   $x^3 - x^2 + 2x - 1$  &  $-23$  \\
$1$ & $-1$ &  $x^3 - 2x^2 + 3x - 1$ &   $-23$ \\
$1$ & $0$ &  $x^3 - 3x^2 + 4x - 1$ &  $-31$  \\
$1$ & $1$ &  $x^3 - 4x^2 + 5x - 1$    &  $-23$  \\

$2$ & $-5$ &  $x^3 - x^2 - x - 1 $   &   $-44$  \\
$2$ & $-4$ &   $x^3 - 2x^2 + x - 1$ &  $-23$  \\
$2$ & $-2$ & $ x^3 - 4x^2 + 5x - 1$ &  $-23$  \\
$2$ & $-1$ &  $x^3 - 5x^2 + 7x - 1$ &  $-44$  \\

\bottomrule
\end{tabular}%
\end{table}

\begin{lemma} \label{lemma:UV_complex}
    Let $U$ be a positive integer and $V \in \Z$ such that $(UV)^2(U+V)^2 - 2(U-V)^3 + 6(U^3-V^3) - 27$ is negative.  Then either $V = -U$ or $(U,V)$ equals one of the 10 following pairs:
    \begin{gather*}
        (1,-2), (1,0), (1, 1), (1, 2), (1, 3), (2, -3), (2, -1), (3, -4), (3, -2), (4, -3).
    \end{gather*} 
\end{lemma}

\begin{proof}
     Define the polynomial 
    \begin{equation*}
        F_U(z) := U^2z^4 + (2U^3 - 4)z^3 + (U^4 - 6U)z^2 + (6U^2)z + 4U^3 - 27 .
    \end{equation*}
    Our problem is therefore to find all integers $V$ such that $F_U(V) < 0$. Let's first consider the polynomials $F_U(z)$ for $U \in \{1, 2, 3, 4\}$:
    \begin{gather*}
        F_1(z) = z^4 - 2z^3 - 5z^2 + 6z - 23, \quad  F_2(z) = 4z^4 + 12z^3 + 4z^2 + 24z + 5, \\
        F_3(z) = 9z^4 + 50z^3 + 63z^2 + 54z + 81, \quad F_4(z) = 16z^4 + 124z^3 + 232z^2 + 96z + 229.
    \end{gather*}
    By doing a standard computation of the real roots of each of the above four polynomials, we observe that $F_1(z) < 0$ for integers $z \in \{-2, -1, 0, 1, 2, 3\}$, $F_2(z) < 0$ for the integers $z \in \{-3, -2, -1\}$, $F_3(z) < 0$ for the integers $z \in \{-4, -3, -2\}$, and $F_4(z) < 0$ for the integers $z \in \{-4, -3\}$.

    Now let's assume $U > 4$.  We claim that the only integer $z$ such that $F_U(z) < 0$ is $z = -U$.  
      As $F_U(z)$ is a quartic polynomial with positive leading coefficient, it has (at most) two local minima.  We now note the following integer values of $F_U(z)$ at $z \in \{-U-1, -U, -U+1, -1, 0, 1\}$:
    \begin{gather*}
        F_U(-U-1) = U^4 - 2U^3 - 5U^2 + 6U - 23, \quad F_U(-U) = -4U^3 - 27, \\
        F_U(-U+1) = U^4 - 6U^3 + 7U^2 + 6U - 31,  \quad       F_U(-1) = U^4 + 2U^3 - 5U^2 - 6U - 23, \\
        F_U(0) = 4U^3 - 27, \quad F_U(1) = U^4 + 6U^3 + 7U^2 - 6U - 31.
    \end{gather*}
    By a standard computation, we observe that $F_U(-U-1) > F_U(-U)$ and $F_U(-U) < F_U(-U+1)$, and also have that $F_U(-1) > F_U(0)$ and $F_U(0) < F_U(1)$ for all $U > 4$.  Thus, the two local minima of $F_U$ must lie within the intervals $(-U-1, -U+1)$ and $(-1, 1)$ respectively.  As $F_U(0) > 0$, we therefore obtain that the only integer value of $z$ such that $F_U(z) < 0$ is $z = -U$.
\end{proof}

\begin{lemma} \label{lemma:n>_complex}
    Let $K$ be a complex cubic number field with negative discriminant and assume $n^3 + an^2 + bn = 0$.  Then either $K = L_n$ and $(\uu, \vv)$ is equivalent to $(\omega_n, -\omega_n + n)$ or $(\uu, \vv)$ is equivalent to a finite number of sporadic solutions.
\end{lemma}

\begin{proof}
    Again, by Corollary~\ref{cor:mem}, we must have that $\disc(f_\uu) < 0$, and Lemma~\ref{lemma:n>_disc} thus reduces the problem to finding all pairs of integers $(U,V)$ such that $(UV)^2(U+V)^2 - 2(U-V)^3 + 6(U^3-V^3) - 27$ is negative. As $n > 0$, we may assume $U > 0$. We therefore apply Lemma~\ref{lemma:UV_complex} to finish the proof.
\end{proof}

\begin{table}[h]
\centering
\caption{In the case $K$ is complex and $n^3 + an^2 + bn = 0$, we tabulate the possibilities for the minimal polynomial $f_{\uu}(x)$, the integer $n$, and number field discriminant $\Delta_K$, given the possible pairs $(U,V)$ from Lemma~\ref{lemma:n>_complex}.}
\label{tab:complexUVs}
\begin{tabular}{@{}cccc@{}}
\toprule
$(U, V)$ & $f_{\uu}(x)$ & $n$ & $\Delta_K$ \\ \midrule

$(1, -2)$  & $x^3 - x - 1$ & 1 & $-23$ \\
$(1, 0)$  & $x^3 - 2x^2 + x - 1$ & $1$ & $-23$ \\
$(1, 1)$  & $x^3 - 3x^2 + 2x - 1$ & $1$ &  $-23$\\
$(1, 2)$  & $x^3 - 4x^2 + 3x - 1$ & $1$ &  $-31$ \\
$(1, 3)$  & $x^3 - 5x^2 + 4x - 1$ & $1$ &  $-23$ \\
$(2, -3)$  & $x^3 - x^2 - 2x - 1$ & $2$ &  $-31$ \\
$(2, -1)$  & $x^3 - 3x^2 + 2x - 1$ & $2$ &  $-23$ \\
$(3, -4)$  & $x^3 - 2x^2 - 3x - 1$ & $3$ &  $-23$ \\
$(3, -2)$  & $x^3 - 4x^2 + 3x - 1$  & $3$ &  $-31$ \\
$(4, -3)$  & $x^3 - 5x^2 + 4x - 1$ & $4$ &  $-23$ \\
\bottomrule
\end{tabular}%
\end{table}

\subsection*{Proof of Theorem~\ref{thm:maincomplex}}
 Let $K$ be a complex cubic field, $n \in \Z$ and units $\uu, \vv \in \co^\times$ such that $\uu + \vv = n$.  As before, if $n = 0$ or $\uu, \vv \in \Q$, then $(\uu, \vv)$ falls inside case (1) in the theorem statement.  

Now assume $n \neq 0$ and $\uu, \vv \nin \Q$.  As with the cyclic cubic case, we can similarly just apply Lemma~\ref{lemma:n12_complex} to handle the $n^3 + an^2 + bn \neq 0$ case and Lemma~\ref{lemma:n>_complex} to handle the $n^3 + an^2 + bn = 0$ case.  Combining all  solutions from both cases gives us the $\numcomplexsporadic$ equivalence classes of sporadic solutions shown in Table~\ref{tab:complexsolutions}.

Finally, to explicitly compute $\NK$ for some complex cubic field $K$, it suffices to determine a list of all integers $b \geq -1$ such that $K$ is isomorphic to $L_b$.  This computation is done for $K = L_a$ all $a = -1, \dots, 1000$,  as shown in Proposition~\ref{prop:computeNKcomplex}. This yields the computation of $\NK$ as shown in (\ref{eq:NKcomplex}).
\qed

\subsection{Table of sporadic solutions}
Here, we give a full table listing all the non-trivial solutions to $\uu + \vv = n$ over all complex cubic fields $K$, shown in Table~\ref{tab:complexsolutions}.  As with the cyclic cubic solutions, we only list solutions up to equivalence.  We recall that a solution $(\uu, \vv)$ to the unit equation $\uu + \vv = n$ over a complex cyclic field $K$ is \emph{non-trivial} if $n \neq 0$, $\uu, \vv \nin \Q$, and $(\uu, \vv)$ is not equivalent to $(\omega_a, -\omega_a + a)$ for any $a \geq -1$.

\input{complex_solutions_table}

\noindent \textbf{Remark}. As with the cyclic cubic case, we note that equivalent sporadic solutions also exist in the complex cubic field $L_{67}$, given that we have $L_{1} = L_{67}$. %

\subsection{Computing $\NK$ for complex cubic fields}
Whilst our proof of Theorem~\ref{thm:maincomplex} gives an explicit classification of all solutions to $\uu + \vv = n$ for units $\uu, \vv \in \co^\times$ in complex cubic fields $K$, in order to explicitly compute the set $\NK$ for some fixed complex cubic field $K$, we must still determine all integers $b$ for which $\omega_b$ lies in $K$.
This therefore requires solving the field isomorphism problem for the family of fields $L_a = \Q(\omega_a)$, for which we can use the following theorem of Hoshi and Miyake:

\begin{theorem}[Hoshi--Miyake {\cite[Theorem~2]{HoshiMiyake}}] \label{thm:hoshimiyake}
Let $n$ and $m$ be two nonzero integers such that the splitting fields of $X^3 + mX + m$ and $X^3 + nX + n$ over $\Q$ coincide.  Then there exists a primitive solution $(x, y) \in \Z^2$ with $y > 0$ to the cubic Thue equation
\begin{equation*}
    x^3 - 2mx^2y - 9mxy^2 - m(2m + 27)y^3 = \lambda
\end{equation*}
for some $\lambda \in \Z$ such that $\lambda^2$ is a divisor of $m^3 (4m + 27)^5$.  In particular, $n$ can be given as
\begin{equation*}
    n = m + \frac{m(4m+27)y(x^2 + 9xy + 27y^2 + my^2)(x^3 - mx^2 y - m^2 y^3)}{(x^3 - 2mx^2y - 9mxy^2 - m(2m + 27)y^3)^2} .
\end{equation*}
    
\end{theorem}

\begin{prop} \label{prop:computeNKcomplex}
For any $a \geq -1$, let $L_a := \Q(\omega_a)$ be the complex cubic field defined by $\omega_a$, where $\omega_a$ is a root of the cubic polynomial $x^3 - ax^2 - 1$.  Then
\begin{equation*}
    \NK = 
    \begin{cases}
           \{1, 2, 3, 4\} & \text{if } K = L_{-1}, \\
           \{1, 2, 3, 67 \}   & \text{if } K = L_1, \\
           \{2, a\}  & \text{if } K = L_a \text{ for some positive } a \leq 1000 \text{ with } a \nin \{1, 67\}, \\
           \{2\} & \text{if } K \neq  L_a \text{ for all } a \geq -1 .
    \end{cases}
\end{equation*}
\end{prop}

\begin{proof}
    Let $K = L_a$ for some $a = -1, \dots,  1000$. We can compute $\NK$ using our proof of Theorem~\ref{thm:maincomplex}. Firstly, we can easily check which of the $\numcomplexsporadic$ sporadic solutions in Table~\ref{tab:complexsolutions} lie in $K$. Therefore, to compute $\NK$, it thus suffices to determine for which integers $b$ does $\omega_b$ lie in $K$.  

    One can check that the minimal polynomial of $-a/\omega_a$ is $X^3 + a^3 X + a^3$, and thus the Galois closure of $L_a$ is isomorphic to the splitting field of $X^3 + a^3X + a^3$.  We can therefore use Theorem~\ref{thm:hoshimiyake} with $n = a^3$ and $m = b^3$ to compute all such integers $b$ such that $\omega_b \in L_a$.  In particular, we do the following,
    \begin{enumerate}
        \item Using Magma, compute all primitive solutions $(x, y) \in \Z^2$ with $y > 0$ to the cubic Thue equation
        \begin{equation*}
            x^3 - 2a^3x^2y - 9a^3xy^2 - a^3(2a^3 + 27)y^3 = \lambda ,
        \end{equation*}
        for each $\lambda \in \Z$ such that $\lambda^2$ is a divisor of $a^9 (4a^3 + 27)^5$.

        \item For each primitive solution $(x,y)$ obtained, compute the possible value for $b$ as 
        \begin{equation*}
            b = \Big( a^3 + \frac{a^3(4a^3+27)y(x^2 + 9xy + 27y^2 + a^3y^2)(x^3 - a^3x^2 y - a^6 y^3)}{(x^3 - 2a^3x^2y - 9a^3xy^2 - a^3(2a^3 + 27)y^3)^2} \Big)^{1/3} .
        \end{equation*}

        \item If $b$ is an integer and $\omega_b$ lies in $L_a$, then add $b$ to $\NK$.
    \end{enumerate}

    By thus solving the above cubic Thue equations for each $a = -1, \dots, 1000$ using Magma, this gives us the sets $\NK$ as claimed in (\ref{eq:NKcomplex}).
\end{proof}

\begin{example}
There are also families of cubic fields, for which the equation $\varepsilon_1+\varepsilon_2=n$ has a solution for several values of $n$. For example, let $\rho$ be a root of the polynomial $x^3+(l-1)x^2-lx-1$ where $l\geq 3$. The fields $\Q(\rho)$ are called Ennola's cubic fields \cite{Ennola}. Then, except for the trivial solutions for $n=0,2$ and $\rho+1-\rho=1$, for $n>1$, we have at least the following four solutions:
\begin{align*}
(-\rho)+(l+\rho)&=l,\\
(1-\rho)+(l+\rho)&=l+1,\\
(l\rho+\rho^2)+(l+2-l\rho-\rho^2)&=l+2,\\
(1+l\rho+\rho^2)+(l+2-l\rho+\rho^2)&=l+3.
\end{align*}
We see that the situation in these fields differs from the simplest cubic fields, where, in almost all cases, we can find a solution only for $n=\pm2,\pm1, 0$.
\end{example}

\section{Further problems} \label{sec:openproblems}

Whilst Theorem~\ref{thm:NKfinite} proves that $\NK$ is finite for number fields $K$ not containing a real quadratic subfield, there are still many open problems one could ask regarding both the quantitative and qualitative behaviour of $\NK$.  In a similar spirit to the conjectures of Vukusic--Ziegler \cite{VZ}, we therefore pose the following open problems:

\begin{itemize}

    \item Given a fixed positive integer $d \geq 1$, what is the average value of $|\NK|$ over all degree $d$ number fields $K$ 
    not containing a real quadratic subfield?

    \item Given a fixed positive integer $d \geq 1$, what is the \emph{maximum} of $|\NK|$ over all degree $d$ number fields $K$ 
    not containing a real quadratic subfield?\footnote{If we let $n_d$ be the maximum of $|\NK|$ over all degree $d$ number fields $K$ not containing a real quadratic subfield, then a quick computational search gives the lower bounds $n_3 \geq 7$, $n_4 \geq 8$, and $n_5 \geq 11$.}

    \item Given a fixed positive integer $d \geq 1$, what is the \emph{minimum} of $|\NK|$ over all degree $d$ number fields $K$?

    \item Does there exist an effective algorithm that accepts as input, a number field $K$ such that $\NK$ is finite, and outputs explicitly the set $\NK$?

    \item Let $\mathcal{N}_{K,m}$ denote the set of positive integers $n \in \Z$ such that $n$ can be represented as the sum of $m$ units in $\co^\times$.  For which number fields is $\mathcal{N}_{K,m}$ finite? If $\mathcal{N}_{K,m}$ is finite, can $\mathcal{N}_{K,m}$ be effectively computed?

\end{itemize}

\bibliographystyle{alpha}  
\bibliography{bib}

\end{document}

%% file: cyclic_solutions_table.tex
\begin{table}[h]
\centering
\caption{List of all non-trivial sporadic solutions to $\uu + \vv = n$ over all cyclic cubic fields $K$, given up to equivalence. All sporadic solutions lie inside the three simplest cubic fields $K_{-1}$, $K_0$, and $K_1$.  Here $\rho$ denotes a root of the given defining polynomial for $K$, and $\Delta_K$ denotes the discriminant of $K$.}
\label{tab:cyclicsolutions}
\begin{tabular}{@{}cccc@{}}
\toprule
\textbf{Defining polynomial for $K$} & \textbf{$\uu$} & \textbf{$\vv$} & \textbf{$n$} \\ \midrule

$f_{-1}(x) = x^3 + x^2 - 2x - 1$  &  
$\rho $  &  $-\rho+1$   &  $1$ \\

($\Delta_K = 7^2$) &   
$\rho^2$  & $-\rho^2 + 1$  & $1$  \\
  &  $\rho^2 - \rho$ & $-\rho^2 + \rho + 1$  & $1$ \\
 &  $-\rho$ & $\rho + 2$  &  $2$ \\
 &  $\rho + 1 $ & $-\rho + 1$  &  $2$ \\
 &  $3\rho^2 + \rho - 6 $ & $-3\rho^2 - \rho + 8$  &  $2$ \\
 &  $\rho^2 $ & $-\rho^2 + 3$  &  $3$ \\
 &  $4 \rho^2 - 5\rho$ & $-4\rho^2 + 5\rho + 3$  &  $3$ \\
 &  $\rho^2 + 2\rho$ & $-\rho^2 - 2\rho + 4$  &  $4$ \\
 &  $\rho^2 - \rho $ & $-\rho^2 + \rho + 5$  &  $5$ \\
 &  $5\rho^2 + 9\rho$ & $-5\rho^2 - 9\rho + 19$   &  $19$ \\
 &  $4\rho^2 - 5\rho$ & $-4\rho^2 + 5\rho + 22$  &  $22$ \\[5mm]

\midrule

  $f_{0}(x) =  x^3 - 3x - 1$  &  
$\rho$ & $-\rho+2$ & 2  \\
($\Delta_K = 3^4$) &   $\rho^2 $  &  $-\rho^2 + 3$ & 3 \\[5mm]

\midrule

    $f_{1}(x) =  x^3 - x^2 - 4x - 1$  &  
$\rho^2 $  &  $-\rho^2 + 2$   &  $2$ \\
($\Delta_K = 13^2$)  &  $\rho^2 - 2$   &  $-\rho^2 + 7$  & $5$ \\
  & $\rho^2$ &  $-\rho^2 + 7$  & $7$ \\[5mm]

\bottomrule
\end{tabular}%
\end{table}

%% file: complex_solutions_table.tex
\begin{table}[h]
\centering
\caption{List of all non-trivial sporadic solutions to $\uu + \vv = n$ over all complex cubic fields $K$, given up to equivalence. All sporadic solutions lie inside the three complex cubic fields $L_{-1}$, $L_1$, and $\Q[x]/(x^3 - x^2 - x - 1)$.  Here $\omega$ denotes a root of the given defining polynomial for $K$, and $\Delta_K$ denotes the discriminant of $K$.}
\label{tab:complexsolutions}
\begin{tabular}{@{}cccc@{}}
\toprule
\textbf{Defining polynomial for $K$} & \textbf{$\uu$} & \textbf{$\vv$} & \textbf{$n$} \\ \midrule

$g_{-1}(x) = x^3 + x^2 - 1$  &  
$\omega$  &  $-\omega + 1$   &  $1$ \\

($\Delta_K = -23$) &   
$\omega^2$ & $-\omega^2 + 1$  & $1$  \\
 &  $\omega^2 + \omega$ & $-\omega^2 - \omega + 1$  & $1$ \\
 &  $-\omega^2 - \omega$ & $\omega^2 + \omega + 1$  &  $1$\\
 &  $\omega^2 + 2\omega + 2$ &  $-\omega^2 - 2\omega - 1$ &  $1$ \\
 & $\omega + 1$  & $-\omega + 1$ &  $2$ \\
 & $\omega^2 + \omega + 1$  &  $-\omega^2 - \omega + 1$  &  $2$ \\
 & $\omega^2 + 2\omega + 1$    & $-\omega^2 - 2\omega + 2$  & $3$ \\
 & $\omega^2 + 2\omega + 2$  &  $-\omega^2 - 2\omega + 2$ &  $4$ \\[5mm]

\midrule

  $g_{1}(x) =  x^3 - x^2 - 1$  &  
$- \omega^2$ & $\omega^2 + 1$ & $1$  \\
($\Delta_K = -31$) & $ \omega^2 - \omega$ & $ -\omega^2 + \omega + 1$ & $1$  \\
& $\omega^2$ & $-\omega^2 + 2$ & $2$  \\
&   $\omega^2 + 1 $  &  $-\omega^2 + 2$ & $3$ \\[5mm]

\midrule

    $x^3 - x^2 - x - 1$  &  
$\omega $  &  $-\omega+ 2$   &  $2$ \\
($\Delta_K = -44$)  &    &    &  \\[5mm]

\bottomrule
\end{tabular}%
\end{table}